\documentclass[12pt]{article}

\usepackage{latexsym,amssymb,amsmath,MnSymbol}
\usepackage[dvips]{graphicx}
\newcommand{\Z}{\mathbb Z}
\newcommand{\N}{\mathbb N}

\newcommand{\Q}{\mathbb Q}

\newcommand{\sma}{\left(\begin{array}}
\newcommand{\fma}{\end{array}\right)}

\newtheorem{lem}{Lemma}[section]
\newtheorem{defn}[lem]{Definition}

\newtheorem{co}[lem]{Corollary}
\newtheorem{thm}[lem]{Theorem}
\newtheorem{prop}[lem]{Proposition}

\newenvironment{proof}{\textbf{Proof.}}{\newline\hspace*{\fill}{$\Box$}\\}

\begin{document}
\title{Acylindrical Hyperbolicity, non simplicity and SQ-universality of 
groups splitting over $\Z$}
\author{J.\,O.\,Button}

\newcommand{\Address}{{% additional braces for segregating \footnotesize
  \bigskip
  \footnotesize

\textsc{Selwyn College, University of Cambridge,
Cambridge CB3 9DQ, UK}\par\nopagebreak
  \textit{E-mail address}: \texttt{j.o.button@dpmms.cam.ac.uk}
}}

\date{}

\maketitle
\begin{abstract}
We show, using acylindrical hyperbolicity, that a finitely
generated group splitting over $\Z$ cannot be simple. We also obtain
SQ-universality in most cases, for instance a balanced group (one where
if two powers of an infinite order element are conjugate then they are
equal or inverse) which is finitely generated and splits over $\Z$
must either be SQ-universal or it is one of exactly seven virtually
abelian exceptions. 
\end{abstract}

\section{Introduction}
An infinite word hyperbolic group can never be simple. Indeed it was shown
in \cite{ol} that a non elementary word hyperbolic group $G$ is SQ-universal,
that is every countable group embeds in a quotient of $G$, and this implies
that $G$ has uncountably many normal subgroups so is very far from being
simple. A generalisation of word hyperbolicity is that of being hyperbolic
relative to a collection of proper subgroups, with non trivial amalgamated
free products $A*_CB$ and HNN extensions $H*_C$ over finite groups $C$ being
examples. It was shown in \cite{aras} that non elementary groups which are
hyperbolic relative to a collection of proper subgroups are SQ-universal too,
thus in particular this result holds for
HNN extensions $H*_C$ with $C$ finite and $C<H$, which was not previously
known. However well before this it was established (see \cite{sch}) that all
non trivial free products are SQ-universal (excluding of course $C_2*C_2$),
and then in \cite{loss} the SQ-universality of
amalgamated free products $A*_CB$, where $C$ is finite and has index greater
than 2 in one factor and at least 2 in the other, was proven. In fact
the bulk of the work here is in showing that these groups cannot be simple.

In contrast all HNN extensions $H*_C$
are easily seen to be non simple because
such a group will always surject $\Z$. However 
it is possible for an amalgamated
free product $A*_CB$ to be simple even if finitely generated, as first
shown in \cite{cm} from 1953 by R.\,Camm where $A,B$ are finitely generated
free groups but $C$ is not finitely generated. More recently the striking
examples of Burger and Mozes in \cite{bm} are simple groups where
$A,B$ are as before, yet $C$ is a subgroup of finite index in both $A$ and
$B$. Now given that $A*_CB$ can never be simple if $C$ is trivial or finite,
but it can be simple for other choices of amalgamated subgroup, this surely
begs the natural question of whether such groups can be simple in the case
where $C\cong\Z$ is infinite cyclic. This seems even more natural in light
of the fact that groups of the form $A*_\Z B$ and $H*_\Z$ have been much
studied, especially in the context of JSJ splittings, but we can find no
instance in the literature of this question of simplicity even being
raised, let alone any direct partial results. (Of course some cases can
be deduced from other work, for instance if $A,B$ are word hyperbolic
groups and the conditions for the Bestvina - Feighn combination theorem
are satisfied then $A*_\Z B$ will also be word hyperbolic, so not simple.
Or if $A,B$ are free groups then $A*_\Z B$ will have a presentation with
more generators than relators so will surject $\Z$.)
We can further ask if any group of the form $A*_\Z B$ is SQ-universal,
outside of a small collection of examples akin to the virtually cyclic
case when $C$ is finite. We can also ask about the SQ-universality of HNN
extensions of the form $H*_\Z$, again allowing a limited list of exceptions.

A more recent development is that of a group being acylindrically hyperbolic,
which is a further generalisation of being hyperbolic relative to proper
subgroups. It is shown in \cite{os} that acylindrical hyperbolicity also
implies SQ-universality (here the definition is set up so that there
is no such thing as an elementary acylindrically hyperbolic group) and
in \cite{asos} a fairly general condition for a finite graph of groups
to be acylindrically hyperbolic is given. However, unlike when $C$ is
finite, we can have groups $A*_\Z B$ or $H*_\Z$ which are SQ-universal
but not acylindrically hyperbolic. Indeed the only criterion we will
need here for showing a group $G$ is not acylindrically hyperbolic is that
of an $s$-normal subgroup $H$ of $G$, which means that $gHg^{-1}\cap H$
is infinite for all $g\in G$. If $G$ is acylindrically hyperbolic then
an $s$-normal subgroup $H$ must be too, so possessing an
infinite cyclic $s$-normal subgroup is an obstruction to acylindrical
hyperbolicity but need not be for SQ-universality, for instance
$F_n\times\Z$ for $n\geq 2$ is SQ-universal and of the form $A*_\Z B$ and
$A*_\Z$ but is not acylindrically hyperbolic.

In this paper we apply the sufficient result in \cite{asos} on 
acylindrical hyperbolicity of the fundamental group of a finite graph
of groups with arbitrary vertex and edge groups to the case where all
edge groups are infinite cyclic (but otherwise no restriction on the
vertex groups), and use this
to show in Corollary \ref{splzco} that if $G$ is finitely generated and
is the fundamental group of a non trivial graph of groups with all edge
groups infinite cyclic then it is not simple; indeed it always has a
non trivial normal subgroup of infinite index so cannot be just infinite
either.

The method of proof here involves distinguishing between balanced and
non balanced elements $x$ of infinite order in a group $G$, where $x$
being balanced means that if $gx^mg^{-1}=x^n$ then $|m|=|n|$. We then
have a version $\Delta_x^G$ of the modular homomorphism of $x$ in $G$,
with the domain being all elements $g\in G$
such that $g\langle x\rangle g^{-1}\cap\langle x\rangle$ is non trivial.
This is a subgroup of $G$ containing the centraliser of $x$ but here
it will only be applied when this subgroup
is all of $G$. In this case, if $x$ is balanced in the
finitely generated group $G$ 
then we obtain an infinite cyclic normal subgroup of $G$ and if not
we have a surjection from $G$ to $\Z$.

As for the question of SQ-universality, it would be good if
we could say that a finitely generated group $G$ of the form
$A*_\Z B$ or $H*_\Z$ is always SQ-universal or one in a list of small 
examples. In Section 3 we first find these small examples, in that we
give all non trivial amalgamated free products $A*_\Z B$ and HNN extensions
$H*_\Z$ which do not contain a non abelian free group. This list
turns out to consist of the soluble Baumslag-Solitar groups and exactly six
other examples. Unfortunately we cannot quite show that all other finitely
generated groups of this form are SQ-universal as it is unclear how to proceed
when mixing balanced and unbalanced elements. However for a finitely
generated group $G$ of the form $A*_{\langle c\rangle} B$ where $c$ has 
infinite order and $G$ is not one of these six small examples, we show in
Theorem \ref{amsq} that $G$ is indeed SQ-universal unless $\langle c\rangle$
is $s$-normal in $G$ and
$c$ is balanced in one of the factors but not balanced in the other. 
Proposition \ref{para} even has
some partial results on SQ-universality in this case.

It is a similar story for finitely generated HNN extensions $G=H*_\Z$ with
stable letter $t$ such that $tat^{-1}=b$ for $a,b$ infinite order elements
of $H$, where Theorem \ref{hnn} says that $G$ is SQ-universal (or $\Z^2$ or
the Klein bottle group) unless $\langle a\rangle$
is $s$-normal in $G$ and $a$ is balanced in
$H$ but where $a^r=b^s$ holds in $H$ with $|r|\neq |s|$ (so that $a$ is
not balanced in $G$). Again we get partial results on SQ-universality in
this leftover case from Proposition \ref{parahnn}. We finish with 
Corollary  \ref{balco} which states that if
every infinite order element of a finitely generated group $G$ is balanced
(which is true of most groups occurring in practice) and $G$ splits
over $\Z$ then either $G$ is $\Z^2$ or one of these six small exceptions,
or $G$ is acylindrically hyperbolic, or $G$ has an infinite cyclic normal
subgroup $Z$ such that $G/Z$ is acylindrically hyperbolic. In the last
two cases $G$ is SQ-universal, so for balanced finitely generated groups
$G$ splitting over $\Z$ we have a complete result: $G$ must be
SQ-universal or is isomorphic to one of exactly seven exceptions.

\section[Acylindrical hyperbolicity of graphs of groups]{Acylindrical 
hyperbolicity of graphs of groups with infinite
cyclic edge groups}
In \cite{asos} a subgroup $H$ of a group $G$ is called {\it weakly malnormal}
in $G$ if there is $g\in G$ such that $gHg^{-1}\cap H$ is finite, and 
{\it $s$-normal} in $G$ otherwise. 
In that paper this concept was introduced in the context of acylindrical 
hyperbolicity,
with a group being acylindrically hyperbolic implying that it is
$SQ$-universal and in particular is not a simple group. The paper gives
sufficient
conditions under which the fundamental group of a finite graph of groups
is acylindrically hyperbolic, which we now review.
First it notes that if a group $G$ is acylindrically hyperbolic then any 
$s$-normal subgroup of $G$
is itself acylindrically hyperbolic (in particular it cannot be a cyclic
subgroup). 

Given a graph of groups $G(\Gamma)$ with connected graph $\Gamma$ and
fundamental group $G$, an edge $e$ is called good if both edge inclusions into
the vertex groups at either end of $e$ give rise to proper subgroups, otherwise
it is bad. A reducible edge is a bad edge which is not a self loop.
Given a finite graph of groups, we can contract the reducible edges one by
one until none are left, whereupon we say $G(\Gamma)$ is reduced.
This process does not affect the fundamental group
$G$ and the new vertex groups will form a subset of the original vertex groups. 
It could be that we are left with a single vertex and no edges, in which case
we say that the graph of groups $G(\Gamma)$ was trivial with $G$ equal to
the remaining vertex group. We then have:
\begin{thm} \label{aomn} (\cite{asos} Theorem 4.17)
Suppose that $G(\Gamma)$ is a finite reduced graph of groups which is 
non trivial and which is not just a single vertex with a single bad edge.
If there are edges $e,f$ of $\Gamma$ (not necessarily distinct) 
with edge groups $G_e,G_f$ and an element $g\in G$ such that
$G_f\cap gG_eg^{-1}$ is finite then $G$ is either virtually cyclic or
acylindrically hyperbolic.
\end{thm}
This immediately gives rise to two corollaries, one for amalgamated free
products and one for HNN extensions:
\begin{co} \label{cortwo}
(i) (\cite{asos} Corollary 2.2)
If $G=A*_C B$ is a non trivial 
amalgamation of any two groups $A,B$ (meaning $C\neq A,B$)
then $G$ is acylindrically hyperbolic (or virtually cyclic) if
$C$ is not $s$-normal in $G$. 

In particular if $C$ is infinite cyclic
then $G$ is acylindrically hyperbolic
exactly when $C$ is not $s$-normal in $G$.\\
(ii) (\cite{asos} Corollary 2.3)
If $G=H*_{tAt^{-1}
=B}$ is a non ascending HNN extension of any group $H$ (meaning
$H\neq A,B$)
with
$\langle a\rangle$ and $\langle b\rangle$ both infinite cyclic
subgroups of $H$
then $G$ is acylindrically hyperbolic if $A$ is not $s$-normal
in $G$.

In particular if $A$ (and thus $B$) is infinite cyclic then
$G$ is acylindrically hyperbolic exactly when $A$ is not $s$-normal
in $G$.
\end{co}

The following is a definition from \cite{wsox}:
\begin{defn}
A group $G$ is called {\bf balanced} 
if for any element $x$ in $G$ of infinite order
we have that $x^m$ conjugate to $x^n$ implies that $|m|=|n|$.\\
Here we will also define:\\ 
A {\bf balanced element} in a group $G$
is an element $x$ in $G$ of infinite order
such that if we have $m,n\in\Z$
with $x^m$ conjugate to $x^n$ in $G$ then $|m|=|n|$. 
\end{defn} 
Thus a group is balanced if and only if all its elements of infinite order
are balanced. Examples of balanced groups are all word hyperbolic groups, 
all 3-manifold
groups, groups acting properly and cocompactly on a CAT(0) space, groups
that are subgroup separable and any subgroup of $GL(n,\Z)$.
As for unbalanced groups, by far the most common examples
are groups containing a Baumslag-Solitar subgroup
$BS(m,n)=\langle a,b\,|\,ba^mb^{-1}=a^n\rangle$
where $|m|\neq |n|$.

Suppose now that we have a cyclic subgroup $\langle x\rangle$ which is
$s$-normal in a group $G$ (and hence $x$ has infinite order).
This means that for any $g\in G$, there
exist non zero integers $m,n$ such that $gx^mg^{-1}=x^n$. Although $m$
and $n$ of course depend on $g$ and are not 
even well defined for a particular $g$, it is easily checked that
the map $\Delta_x^G$ from $G$ to the non zero multiplicative rationals
$\Q^*$ given by $\Delta_x^G(g)=m/n$
is well defined and is even a homomorphism, which we call 
the {\bf modular homomorphism} of $x$ in $G$ (in line with other cases
such as in generalised Baumslag-Solitar groups). In particular the
element $x$ is balanced in $G$ if and only if the image $\Delta_x^G(G)$
of the modular homomorphism is contained in $\{\pm 1\}$.
  
We can now use the above to show that a finitely generated
simple group cannot split over $\Z$.
Indeed we will actually show that a finitely generated group splitting
over $\Z$ has a non trivial normal subgroup of infinite index. 
\begin{thm} \label{splz}
If $G(\Gamma)$ is a non trivial graph of groups with finitely generated
fundamental group $G$ and where all edge groups are
infinite cyclic 
then $G$ is either:\\
(i) acylindrically hyperbolic or\\
(ii) has a homomorphism onto $\Z$ or\\
(iii) has an infinite cyclic normal subgroup.\\
\end{thm}
\begin{proof} 
As $G$ is finitely generated,
we can assume that $\Gamma$ is a finite graph, and then we can reduce
so that either $G=A*_CB$ for $\Z\cong
\langle c\rangle=C\neq A,B$ or $G=H*_{tAt^{-1}=B}$ for $A,B\cong\Z$.
But in the latter case of an HNN extension, $G$ surjects to $\Z$
anyway.
We now apply Corollary \ref{cortwo} in the case of an amalgamated
free product to obtain acylindrical hyperbolicity
of $G$, unless $\langle c\rangle$ is $s$-normal in $G$ which we now
assume.

First suppose that $c$ is also balanced in $G$.
Then for all $g\in G$ we can find
integers $k>0$ and $l\neq 0$, both depending on $g$,
such that $gc^kg^{-1}=c^l$. But $c$ being balanced in $G$ means that
$l=\pm k$. So on taking a finite generating set
$g_1,\ldots ,g_s$ for $G$, we have for each $1\leq i\leq s$ an integer
$k_i$ with $g_ic^{k_i}g_i^{-1}=c^{\pm k_i}$. Consequently we can find
a common power $p$ that
works for all of this set and hence for all of $G$, thus $\langle c^p\rangle$
is normal in $G$. 

If however $c$ is not balanced in $G$ then the image of the
modular homomorphism $\Delta_c^G$ of $c$ in $G$ will not be
contained in $\{\pm 1\}$. But then
$|\Delta_c^G|$ provides a homomorphism from $G$ to the positive rationals
$\Q^+$
which is non trivial, thus the image is an infinite torsion free abelian
group which is finitely generated, thus must be $\Z^n$.
\end{proof} 

\begin{co} \label{splzco} 
If $G$ is a finitely generated group splitting non trivially
over $\Z$ then $G$ has infinitely many non trivial normal subgroups, at
least one of which has infinite index in $G$. Consequently $G$ is not
simple or just infinite and nor is any finite index subgroup of $G$.
\end{co}
\begin{proof}
On application of Theorem \ref{splz}, either $G$ is SQ universal in which
case it has uncountably many normal subgroups
(only countably many of which can have finite index as $G$ is finitely
generated), or it surjects to $\Z$ (therefore to every finite cyclic
group) and this kernel is non trivial as $\Z$ does not split non trivially
over $\Z$, or there is $p>0$ such that $\langle c^p\rangle$ is an infinite
cyclic normal subgroup of $G$, hence this is a proper subgroup, as are
the distinct proper normal subgroups $\langle c^{np}\rangle$ for $n\in\N$,
with their intersection of infinite index and normal too.
The same holds for any finite index subgroup of $G$ as can been in a
variety of ways, not least because it also splits over $\Z$.
\end{proof}

Note: it is quite possible in the amalgamated free product case
that $G$ has no proper finite index subgroups,
for instance if $A$ and $B$ have no proper finite index subgroups then
nor will $A*B$ or $A*_CB$.
 
\section{SQ-universality}   

We can now ask whether the stronger property of being SQ-universal
holds for finitely generated groups splitting over $\Z$, given that
this is true for all acylindrically hyperbolic groups. However this will fail
for the groups mentioned below in Proposition \ref{excpt},
because being SQ universal implies containing a
non abelian free group. Hence we first consider this case by
using the following well known proposition.
\begin{prop} \label{nofree} \hfill\\
(i) If $G=A*_CB$ is a non trivial amalgamated free product so that
$C\neq A,B$ then $G$ contains a non abelian free group, unless neither
of $A$ and $B$ do and $C$ has index 2 in both $A$ and $B$, in which case
$G$ does not.\\
(ii) If $G=H*_{tAt^{-1}=B}$ is an HNN extension then $G$ contains a non abelian
free group, unless $H$ does not and at least one of $A$ and $B$ is equal
to $H$, in which case $G$ does not.
\end{prop}

Thus now we need to see which groups occur in the above when $C$ is isomorphic
to $\Z$, as they cannot possibly be SQ-universal despite splitting
non trivially over $\Z$. 
\begin{prop} \label{excpt}
If the group $G$ splits as an HNN extension over $\Z$ but does not contain a
non abelian free group then $G$ is isomorphic to one of the soluble
Baumslag - Solitar groups $BS(1,n)$ for $n\in\Z\setminus \{0\}$.\\
\hfill\\
If the group $G$ splits as a non trivial amalgamated free product 
over $\Z$ but does not contain a
non abelian free group then $G$ is isomorphic to one of the following
(mutually non isomorphic) six groups:\\
(i) The group $\langle a,b|a^2=b^2\rangle$ which is the Klein bottle group
$K$ (and is also $BS(1,-1)$, so is the only group here splitting over $\Z$
both as an HNN extension and an amalgamated free product).\\
(ii) The group
$\langle s,t,c|[s,t],c^2,csc^{-1}=t,ctc^{-1}=s\rangle$\\
which is $\Z^2\rtimes C_2$, where $C_2$ swaps this free generating set
of $\Z^2$.\\
(iii) The group
$\langle a,b,c|a^2=b^2,c^2,cac^{-1}=b,cbc^{-1}=a\rangle$\\
which is $K\rtimes C_2$ where the generators for $K$ in (i) are swapped.\\
(iv) The group
$\langle c,d,z|c^2,d^2,[c,z],[d,z]\rangle$\\
which is $(C_2*C_2)\times\Z$.\\
(v) The group
$\langle t,c,d|c^2,d^2,ctc^{-1}=t^{-1},dtd^{-1}=t\rangle$\\
which is $\Z\rtimes (C_2*C_2)$ where one $C_2$ factor inverts
$\Z$ and one fixes $\Z$.\\
(vi) The group
$\langle s,t,c|[s,t],c^2,csc^{-1}=s^{-1},ctc^{-1}=t^{-1}\rangle$\\
which is $\Z^2\rtimes C_2$ where $C_2$ inverts the elements of $\Z^2$.
\end{prop} 
\begin{proof} The HNN extension case follows directly from Proposition
\ref{nofree} (ii) on putting $A$ or $B$ equal to $\Z$. In the amalgamated
free product case we have that $A$ and $B$ contain $\Z$ with index 2,
meaning that they are isomorphic to (a) $\Z$ itself, (b) $\Z\times C_2$
or (c) $C_2*C_2$. Now $\Z$ and $C_2*C_2$ each have a unique index 2 subgroup
that is isomorphic to $\Z$ whereas $\Z\times C_2$ has two, but these
are equivalent under an automorphism. Therefore we can write out standard
presentations for all three groups and then for the six possible
amalgamations, which we have done above in the order
(a)-(a),(a)-(b),(a)-(c),(b)-(b),(b)-(c),(c)-(c). We have also performed
some tidying up of the resulting presentations, and calculated
the abelianisation of these groups and their index 2 subgroups which
distinguishes these six groups, as well as distinguishing the last five
from any Baumslag - Solitar group.
\end{proof}

We now consider all other finitely generated
free products amalgamated over $\Z$, where we can prove SQ-universality in
most cases.
\begin{thm} \label{amsq}
Suppose that $G$ is finitely generated and is equal to a non trivial
amalgamated free product
$A*_CB$ for $C=\langle c\rangle\cong\Z$. If $\langle c\rangle$ is not
$s$-normal in $G$, or $\langle c\rangle$ is $s$-normal in $G$ and
$c$ is balanced in $G$
but $G$ is not one of the six groups listed in Proposition \ref{excpt},
or if $\langle c\rangle$ is $s$-normal in $G$ but $c$ is not balanced
in $A$ and not balanced in $B$ then $G$ is SQ-universal.
\end{thm}
\begin{proof} 
By Corollary \ref{cortwo} (i) we know that $G$ is acylindrically hyperbolic
and hence SQ-universal unless $\langle c\rangle$ is $s$-normal in $G$, which
will be assumed for the rest of the proof.
First suppose that $c$ is balanced in $G$, so as in the proof of 
Theorem \ref{splz} we have $p>0$ such that $\langle c^p\rangle$ is normal 
in $G$, and hence
in $A$ and $B$. This means that on quotienting both $A$ and $B$ by this
infinite cyclic normal subgroup to obtain $\overline{A}$ and
$\overline{B}$, in each of which the image $\overline{c}$ has order $p$,
our amalgamated free product $A*_CB$ factors through
$\overline{A}*_{\overline{C}}\overline{B}$ where $\overline{C}=
\langle\overline{c}\rangle$ is cyclic of order $p$. Moreover
$\langle c^p\rangle$
is contained in $C$, so that the indices $[\overline{A}:\overline{C}]=
[A:C]$ and $[\overline{B}:\overline{C}]=[B:C]$ are unchanged. 
But $[A:C],[B:C]>1$ by assumption and at least one of these indices
is greater than two (or else we are back in the case of Proposition
\ref{excpt}),
so by \cite{loss} we have that $\overline{A}*_{\overline{C}}\overline{B}$ is 
SQ-universal and hence $G$ is also. 

Now we assume that $c$ is not balanced in $A$, nor in $B$. This
means that the modular homomorphisms $|\Delta_c^A|$ from $A$
to $\Q^+$ and $|\Delta_c^B|$ from $B$ to $\Q^+$ are both
maps to infinite torsion free abelian
groups with $\langle c\rangle$ in the kernel, thus we can put them together
to obtain a homomorphism from $G$ to the free product of these infinite
abelian groups, thus here $G$ is SQ-universal.
\end{proof}

Thus we are left with $c$ being balanced in $A$ but not in $B$
(or vice versa), whereupon
we cannot always say here that $A*_CB$ is SQ-universal. However there
are subcases where we can obtain this conclusion so we look at this in a
little more detail. By D.\,E.\,Cohen's comment (\cite{coh} Section 3)
we have that in an amalgamated free product $A*_CB$, if $G$ and $C$ are
finitely generated then so are $A$ and $B$.
\begin{prop} \label{para}
Suppose that $G=A*_CB$ is a non trivial amalgamated free
product with $C=\langle c\rangle$ infinite cyclic and $G$ finitely
generated. Suppose also that $\langle c\rangle$ is $s$-normal in $G$
(thus in $A$ and in $B$)
and $c$ is balanced in $A$ but not in $B$, so that (as 
Cohen's comment
tells us $A$ is finitely generated) there is $p>0$ with $\langle c^p\rangle$
normal in $A$. 

If there exists an integer $k>0$ such that\\
(i) the normal closure $\llangle c^k\rrangle^A$ of the element
$c^k$ in $A$ does not contain any of the elements
$c,c^2,\ldots ,c^{k-1}$, with the same holding for the normal closure
$\llangle c^k\rrangle^B$ of $c^k$ in $B$, and\\
(ii) the quotient of $A$ by the normal closure $\llangle c^k\rrangle^A$
has order greater than $k$\\
then $G$ is SQ-universal.

In particular if $\llangle c\rrangle^A$ is not equal to $A$, so that
(ii) holds for $k=1$ with (i) then holding automatically, or if there
is $k$ which is a multiple of $p$ such that (i) holds just for
$\llangle c^k\rrangle^B$, then $G$ is SQ-universal.        
\end{prop}
\begin{proof} Condition (i) 
implies that in both quotients $A/\llangle c^k\rrangle^A$ and
$B/\llangle c^k\rrangle^B$, the element $c$ has order exactly $k$. 
If so then $B/\llangle c^k\rrangle^B$ is infinite because it
surjects $B/\llangle c\rrangle^B$ which itself surjects an
infinite abelian group using $|\Delta_c^B|$, as
$c$ is not balanced in $B$. In particular 
the image of $\langle c\rangle$ in this quotient $B/\llangle c^k\rrangle^B$
of $B$ has infinite index.
Hence the amalgamated free product
$(A/\llangle c^k\rrangle^A)*_{\langle c\rangle}(B/\llangle c^k\rrangle^B)$ is
a quotient of $A*_CB$ and the former is SQ-universal by \cite{loss}
provided only that $\langle c\rangle$ has index greater than 1 in
$A/\llangle c^k\rrangle^A$, which is condition (ii).

As for the particular cases mentioned, taking $k=1$ gives us the free
product $(A/\llangle c\rrangle^A)*(B/\llangle c\rrangle^B)$. If the
left hand factor is non trivial then this will be a non
trivial free product not equal to $C_2*C_2$, and therefore is
SQ-universal. Meanwhile if
$c$ has order exactly $k$ in $B/\llangle c^k\rrangle^B$ for $p$
dividing $k$ then we know $c$ also has order $k$ in $A/\llangle c^k\rrangle^A$ 
because $\langle c^p\rangle$ and then $\langle c^k\rangle$ is normal in $A$,
with $[A/\langle c^k\rangle:C/\langle c^k\rangle]=[A:C]>1$ so (ii) holds
as well.
\end{proof}

We now move to HNN extensions.
\begin{thm} \label{hnn}
If $G$ is a finitely generated group which is an HNN extension
$H*_{tAt^{-1}=B}$ with base $H$ and stable letter $t$ conjugating the
infinite cyclic subgroup $A=\langle a\rangle$ of $H$ to $B=\langle b\rangle$
via $tat^{-1}=b$ then $G$ is SQ-universal or $\Z^2$ or the Klein bottle
group, with the possible exception
of when $\langle a\rangle$ is $s$-normal in $H$ and
there exist integers $r,s$ with $a^r=b^s$ for $|r|\neq |s|$
but $a$ (equivalently $b$) is balanced in $H$.
\end{thm}
\begin{proof}
By Corollary \ref{cortwo} we have that $G$ is acylindrically hyperbolic and
hence SQ-universal unless $\langle a\rangle$ is $s$-normal in $G$. This
certainly implies that $\langle a\rangle$ is $s$-normal in $H$, but also
that there are non zero integers $r,s$ with
$a^r=b^s$ because otherwise $A\cap B$ is trivial with $A$
conjugate to $B$ in $G$, thus $A$ is not $s$-normal in $G$.
Now our HNN
extension $G$ factors through the HNN extension
with base $H/\llangle a,b\rrangle^H$ and trivial edge subgroups, which is
the free product $(H/\llangle a,b\rrangle^H)*\Z$ and thus is SQ-universal
provided the first factor is non trivial. But if $a$ is not balanced in $H$
then $a$, and hence also $b$ because of the relation $a^r=b^s$, 
will be in the kernel of the modular
homomorphism $|\Delta_a^H|$ which here has infinite image in $\Q^+$,
so $H/\llangle a,b\rrangle^H$ is non trivial.

Now say $a$ is balanced in $H$ but $|r|=|s|$. Then as $a$ is $s$-normal
in $H$ too, we have $p>0$ such that
$\langle a^p\rangle$, and hence $\langle a^k\rangle$ for any multiple
$k$ of $p$, is normal in $H$ as before by using
D.\,E.\,Cohen's comment
applied to HNN extensions (namely $G$ and $A$ being finitely generated
imply the base $H$ is finitely generated). Hence
$\langle a^{p|r|}\rangle=\langle b^{p|r|}\rangle$ is
normal in $H$ with $a$ and $b$ both having order exactly $p|r|$ in this
quotient, so $G$ factors through the HNN extension
$(H/\langle a^{p|r|}\rangle)*_{tAt^{-1}=B}$ where now $A$ and $B$ both
map to
finite cyclic groups of order $p|r|$. This HNN extension is
SQ-universal by \cite{aras} Corollary 1.3
unless we have $A/\langle a^{p|r|}\rangle
=B/\langle b^{p|r|}\rangle=H/\langle a^{p|r|}\rangle$,
whereupon we must have had $A=B=H\cong\Z$ initially. This means that
$G$ was $\langle t,h|tht^{-1}=h^{\pm 1}\rangle$, so is $\Z^2$
or the Klein bottle group. 
\end{proof}

As any HNN extension over $\Z$ excluded by Theorem \ref{hnn}   
will have the edge generator $a$ (equivalently $b$) balanced in
the base group but not balanced in the HNN extension because
we would have
$ta^st^{-1}=a^r$ for $|r|\neq |s|$, we see that (just as in Theorem
\ref{amsq}) it is the ``mixed'' case which is
troublesome in proving SQ-universality. Nevertheless here we can still
obtain partial results similar to those in Proposition \ref{para}. 
\begin{prop} \label{parahnn}
Suppose that
$G$ is a finitely generated group which is an HNN extension
$H*_{tAt^{-1}=B}$ with base $H$ and stable letter $t$ conjugating the
infinite cyclic subgroup $A=\langle a\rangle$ of $H$ to $B=\langle b\rangle$
via $tat^{-1}=b$, such that $\langle a\rangle$ is $s$-normal in $H$ and
there are integers $r,s$ with $a^r=b^s$ for $|r|\neq |s|$ 
but $a$ (equivalently $b$) is balanced in $H$.

If there exists an integer $k>0$ such that\\
(i) the normal closure 
$\llangle a^k,b^k\rrangle^H=
\llangle a^k\rrangle^H\llangle b^k\rrangle^H$ of the elements
$a^k,b^k$ in $H$ does not contain any of the elements
$a,a^2,\ldots ,a^{k-1}$ or $b,b^2,\ldots ,b^{k-1}$, and\\
(ii) the quotient of $H$ by
$\llangle a^k,b^k\rrangle^H$ has order greater than $k$\\
then $G$ is SQ-universal.

In particular if $\llangle a,b\rrangle^H$ is not equal to $H$, so that
(ii) holds for $k=1$ with (i) then holding automatically, 
then $G$ is SQ-universal.        
\end{prop}
\begin{proof}
On setting $N=\llangle a^k,b^k\rrangle^H$ we see that $a$ and $b$ both have
exact order $k$ in $H/N$ so the HNN extension $(H/N)*_{tat^{-1}=b}$ is
well defined, as $a$ and $b$ both have order $k$ in $H/N$, and it is a
quotient of $G$. Moreover neither $a$ nor $b$ generate $H/N$ from (ii)
so $G$ is SQ-universal by \cite{aras} Corollary 1.3.
\end{proof}

To finish,
if $G$ is finitely generated, splits over $\Z$ and is a balanced group
then we have a complete result. Indeed we have a trichotomy (reading
like \cite{asos} Theorem 2.8 but
there the groups being considered are all subgroups of fundamental groups
of compact 3-manifolds) which says
that $G$ is acylindrically hyperbolic, or is so on quotienting
out by an infinite cyclic normal subgroup (and therefore is SQ-universal
because this is preserved by prequotients), or is one of a very few small
exceptions. Indeed we only need the edge group generators to be balanced
for this to hold.
\begin{co} \label{balco}
Suppose that $G$ is a finitely generated group which equals the 
fundamental group of
a non trivial graph of groups with infinite cyclic edge groups
and where the edge group generators are balanced in $G$.
Then one of the following mutually exclusive cases occurs:\\
(i) $G$ is acylindrically hyperbolic\\
(ii) $G$ has an infinite
cyclic normal subgroup $Z$ such that $G/Z$ is relatively hyperbolic
but not virtually cyclic and so is acylindrically hyperbolic\\
(iii) $G$ is isomorphic to $\Z^2$ or one of the six groups
listed in Proposition \ref{excpt}.\\
In particular $G$ is SQ-universal
or virtually abelian.
\end{co}
\begin{proof} 
On applying Theorem \ref{splz} we see that
if $G$ is not acylindrically hyperbolic
then either it is an HNN extension, or it is an amalgamated free
product with an infinite cyclic normal subgroup $Z$ whose generator $c^p$
is a power of the edge group generator $c$, because we are in the case
where $c$ is balanced in $G$.
In the latter case we can
apply Theorem \ref{amsq} with $\langle c\rangle$ being $s$-normal and 
$c$ balanced in $G$ to conclude that either $G/Z$ is an
amalgamation over a finite subgroup which is not virtually cyclic
but which is hyperbolic relative to the vertex groups, or
$G$ is one of the six groups listed in Proposition \ref{excpt}.

If now $G$ is an HNN extension we use Theorem \ref{hnn}, whereupon 
$\langle a\rangle$ $s$-normal in $G$ and $a$ balanced in $G$ implies
that $a$ is balanced in $H$ and $a^r=b^s$ for
$|r|=|s|$. Here we conclude that the infinite cyclic subgroup 
$C=\langle a^{p|r|}\rangle$ is normal not just in the base but in $G$ too
as $a^{p|r|}=b^{\pm p|r|}$ with $ta^{pr}t^{-1}=b^{\pm pr}$. Moreover
the quotient $G/Z$ is again relatively hyperbolic
but not virtually cyclic, or $G$ is $\Z^2$ or the Klein bottle group
with the latter being the first group listed in Proposition \ref{excpt}.  
\end{proof}

\Address
\end{document}